\documentclass[a4paper, 12pt]{amsart}

\usepackage{amsmath,amsfonts,amssymb,amscd}
\usepackage{verbatim}
\usepackage{stmaryrd}
\usepackage{enumitem}
\usepackage{hyperref,mdwlist}
\usepackage[english]{babel}
\usepackage{latexsym}
\usepackage{amsopn}
\usepackage{mathrsfs}
\usepackage{float}
\usepackage{fullpage}
\usepackage{tikz}

\usepackage{graphicx}
\usepackage{tikz-cd}
\usepackage{csquotes}
\usepackage[backend=biber,style=alphabetic,maxbibnames=99]{biblatex}
\usepackage[T1]{fontenc}
\usepackage[utf8]{inputenc}
\usepackage{amssymb,amsthm,mathtools}
\usepackage{microtype}
\hypersetup{hidelinks}
\usepackage{cleveref}

\newtheorem{theorem}{Theorem}[section]
\newtheorem{proposition}[theorem]{Proposition}
\newtheorem{lemma}[theorem]{Lemma}
\newtheorem{corollary}[theorem]{Corollary}
\newtheorem{remark}[theorem]{Remark}
\newtheorem{example}[theorem]{Example}

\newcommand{\FFLV}{\mathrm{FFLV}}
\newcommand{\Mak}{\mathrm{Mak}}
\newcommand{\Str}{\mathrm{Str}}
\newcommand{\ZZ}{\mathbb Z}
\newcommand{\RR}{\mathbb R}

\newcommand{\ww}{\mathbf w}

\newcommand{\eps}{\varepsilon}

\title[Makhlin polytopes and Demazure string polytopes]{Makhlin polytopes are Demazure string polytopes}
\author{Ghislain Fourier}
\address{Chair of Algebra and Representation Theory, RWTH Aachen University, Pontdriesch 10--16, 52062 Aachen, Germany}
\email{fourier@art.rwth-aachen.de}


\begin{document}

\begin{abstract}
We show that Makhlin's polytopes in type $B_n$ are unimodularly equivalent to string polytopes for Demazure modules in type $B_{2n-1}$. The proof passes through type $C$, combining Makhlin's diagonal comparison with symplectic FFLV polytopes and the $B/C$ similarity for string cones.
\end{abstract}

\maketitle

\section{Introduction}

Let $\mathfrak{g}$ be a finite-dimensional semisimple complex Lie algebra with fixed triangular decomposition $\mathfrak{g} = \mathfrak{n}^+ \oplus \mathfrak{h} \oplus \mathfrak{n}^-$. The PBW filtration on the universal enveloping algebra of its negative part induces a filtration on highest weight modules. This leads naturally to degenerations of flag varieties. In type $A$, these varieties were introduced and studied by Feigin \cite{Fei11,Fei12}; the symplectic version was considered by Feigin, Finkelberg, and Littelmann \cite{FFL14}. 
A first surprising link with Schubert geometry was established by Cerulli Irelli and Lanini \cite{CL15}: the PBW degenerate flag varieties in types $A$ and $C$ are Schubert varieties in partial flag varieties of the same type and higher rank. A representation-theoretic explanation was given by the same authors together with Littelmann \cite{CILL16}, realizing the associated graded modules as higher-rank Demazure modules. 
For PBW-degenerated Demazure modules and related monomial bases, see also \cite{Fou16,BFK23}. 
The type $A$ picture was subsequently enlarged to linear degenerations and partially degenerate Lie algebras by Cerulli Irelli, Fang, Feigin, Reineke, and the author \cite{CIFFR17}. 
More recently, Enugandla, Fang, Steinert and the author used Dynkin abelianisations to extend the Demazure module and Schubert variety realization to a large class of degenerations in all classical types \cite{EFFS24}.

On the combinatorial side, the PBW bases in types $A$ and $C$ are parametrized by the FFLV polytopes, defined through Dyck path combinatorics (symplectic Dyck paths in type $C$) by Feigin, Littelmann and the author \cite{FFL11A,FFL11C}. Cleusters, Lerner and the author \cite{CFL24} provided the combinatorial shadow of the module isomorphism and the isomorphism of varieties: the FFLV polytopes are unimodularly equivalent to string polytopes of the higher-rank Demazure modules appearing above. For type $B$, an ordinary abelian FFLV model is not available. Makhlin constructed instead a weighted PBW filtration and FFLV-type polytopes whose lattice points parametrize monomial bases \cite{Mak19}. Their defining inequalities are again indexed by Dyck paths, but the normalization differs from the symplectic case. For other polyhedral constructions related to PBW bases in type $B$, see \cite{BK19,MM25}.

The aim of this note is to obtain the combinatorial shadow in type $B$. We do not repeat the type $A$ and $C$ analysis in type $B$. Instead, we compose three existing maps: Makhlin's diagonal map $2D_n^{-1}$ from the type $B$ polytope to the corresponding symplectic FFLV polytope \cite[Remark~2.3]{Mak19}, the affine unimodular map $T_{n,\mu}$ from \cite{CFL24}, and the inverse $B/C$ similarity map $\Gamma_n^{-1}$ from Cho, Fujita, and Lee \cite{CFL25}. The only additional points are to verify that the latter also respects the highest-weight inequalities and that the resulting composition is integral. The construction is summarized in the following diagram.

\begin{figure}[ht]
\centering
\begin{tikzcd}[column sep=huge,row sep=large]
P_{\Mak}^{B_n}(\lambda)
  \arrow[r,"2D_n^{-1}"]
  \arrow[d,dashed,"\Phi_{n,\lambda}"']
&
P_{\FFLV}^{C_n}(\mu)
  \arrow[d,"T_{n,\mu}"]
\\
Q_{\ww_n}^{B_{2n-1}}(\widetilde\lambda)
&
Q_{\ww_n}^{C_{2n-1}}(\widetilde\mu)
  \arrow[l,"\Gamma_n^{-1}"']
\end{tikzcd}
\caption{The three maps used in the proof.}
\label{fig:three-maps}
\end{figure}

Let $\ww_n$ be the reduced word used in the type $C$ construction of \cite{CFL24}; it is recalled explicitly in Section~\ref{sec:preliminaries}. The main result is the following.

\begin{theorem}\label{thm:main}
Let $\lambda=\sum_{i=1}^n a_i\omega_i$ be a dominant integral weight of type $B_n$ and set
\[
           \widetilde\lambda=\sum_{i=1}^n a_i\widetilde\omega_{2i-1}
\]
in type $B_{2n-1}$. Then Makhlin's polytope $P_{\Mak}^{B_n}(\lambda)$ is affinely unimodularly equivalent to the Demazure string polytope
\[
             Q_{\ww_n}^{B_{2n-1}}(\widetilde\lambda).
\]
The equivalence can be chosen compatibly with addition of dominant weights.
\end{theorem}

The diagram also explains the main point of the proof. The map from Makhlin's polytope to the type $C$ FFLV polytope and the $B/C$ similarity on the string side contain the same diagonal rescaling in opposite directions. Hence these non-unimodular factors cancel in the composition. It remains to check that the resulting conjugate of the type $C$ unimodular map, as well as the translation part, is integral. This is done in Sections~\ref{sec:similarity} and \ref{sec:proof}.

The compatibility with addition gives an isomorphism of the corresponding multigraded lattice-point semigroups. In particular, the additional component $2\omega_n$ occurring in Makhlin's Minkowski-type argument \cite[Lemma~5.1]{Mak19} is transported to the string-polytope side.

A PBW-type polyhedral model in type $G_2$ was constructed by Gornitskii \cite{Gor15}. Computations in \cite{CFL24} indicate that Gornitskii's polytope is not carried to the corresponding Demazure string polytope. In a forthcoming preprint, we consider the reverse direction: starting from a Demazure string basis, we pull the basis back to the PBW side via \cite{Enu22} and study the resulting polytope and its lattice-point semigroup.

The paper is organized as follows. In Section~\ref{sec:preliminaries} we recall the three polyhedral constructions. Section~\ref{sec:similarity} extends the $B/C$ similarity of \cite{CFL25} from string cones to the corresponding highest-weight cuts and compares the two diagonal maps. Theorem~\ref{thm:main} is proved in Section~\ref{sec:proof}. Section~\ref{sec:semigroup} contains the consequences for lattice points and semigroups.

\medskip
\noindent\textbf{Acknowledgements.} The author gratefully acknowledges financial support by the DFG -- Project-ID 286237555 -- TRR 195. ChatGPT (OpenAI) was used for language editing and for computations in examples.

\section{Preliminaries}\label{sec:preliminaries}

We recall the notation needed below; see \cite{Mak19,CFL24,CFL25} for details.

\subsection{The common Demazure word}

Let $n\geq2$ and set
\[
                           m=2n-1,\qquad N=n^2.
\]
Define
\begin{equation}\label{eq:intro-word}
 \ww_n=\tau_m\tau_{m-1}\cdots\tau_{n+1}
              \sigma_n\sigma_{n-1}\cdots\sigma_1,
\end{equation}
where
\[
 \tau_j=s_js_{j+1}\cdots s_m,
 \qquad
 \sigma_j=s_js_{j+1}\cdots s_{2j-1}.
\]
This is the reduced word from \cite{CFL24}. Since $W(B_m)=W(C_m)$ as Coxeter groups, it is reduced in both types and has length $N$. We use tildes for the fundamental weights in rank $m$.

For $X=B,C$, let $\eta$ be dominant and let $\mathbf i=(i_1,\ldots,i_L)$ be reduced. Choose a reduced extension
\[
              \widehat{\mathbf i}=(i_1,\ldots,i_L,\ldots,i_M)
\]
to the longest element. 
The Demazure string polytope $Q_{\mathbf i}^{X}(\eta)$ is the projection of
$Q_{\widehat{\mathbf i}}^{X}(\eta)$ to the first
$L$ coordinates \cite{Cal02}.
Its lattice points parametrize the Demazure crystal.
We use the standard lattice in string coordinates.

For a reduced word of the longest element, the highest-weight inequalities can be written in the form
\begin{equation}\label{eq:littelmann}
 x_k\leq \eta_{i_k}-\sum_{j>k}a_{i_k,i_j}x_j,
 \qquad
 \eta_i=\langle\eta,\alpha_i^\vee\rangle,
\end{equation}
where $A=(a_{ij})$ is the Cartan matrix; see \cite{Lit98,BZ01}. Throughout, we use the convention
\[
        a_{ij}=\langle\alpha_j,\alpha_i^\vee\rangle .
\]

\subsection{The type $C$ equivalence}

We use the following result from \cite{CFL24}. Let
\[
                  \mu=\sum_{i=1}^n b_i\omega_i^{C_n},
 \qquad
                  \widetilde\mu=\sum_{i=1}^n b_i\widetilde\omega_{2i-1}^{C_m}.
\]
Denote the type $C_n$ FFLV polytope by $P_{\FFLV}^{C_n}(\mu)$.

\begin{theorem}[{\cite[Theorem~3.1]{CFL24}}]\label{thm:CFL}
There exists an affine map
\[
                    T_{n,\mu}(x)=X_nx+t_{n,\mu}
\]
with $X_n\in GL_N(\ZZ)$ and $t_{n,\mu}\in\ZZ^N$, linear in $\mu$, such that
\[
 T_{n,\mu}\bigl(P_{\FFLV}^{C_n}(\mu)\bigr)
      =Q_{\ww_n}^{C_m}(\widetilde\mu).
\]
\end{theorem}

We use the coordinates $e_{a,b}$ of \cite[Section~4]{CFL24}, with alphabet
\[
 1<2<\cdots<n=\bar n<\overline{n-1}<\cdots<\bar1.
\]
The coordinates
\[
                         e_{a,\bar a},\qquad 1\leq a\leq n,
\]
correspond to the long roots $2\eps_a$ of $C_n$; denote their set by $\Delta_n$. By \cite[Notation~4.10]{CFL24} these labels index simultaneously the FFLV coordinates and the string coordinates of $\ww_n$. 
We use this identification throughout and regard $X_n$, $D_n$ and $\Gamma_n$ as endomorphisms of one and the same space $\RR^N$; in particular the conjugation in \eqref{eq:Uc} below is meaningful.

Let $X_n$ be the linear map from
\cite[Definition~4.1]{CFL24}. Thus
\begin{equation}\label{eq:Xdiag}
 X_n(e_{a,\bar a})
  =-\left(e_{a,\bar a}+2\sum_{c<a}e_{c,\bar a}\right).
\end{equation}
The coefficients of $X_n$ lie in $\{0,-1,-2\}$ and
$|\det X_n|=1$ \cite[Proposition~4.4]{CFL24}.

\begin{lemma}\label{lem:translation}
For $t_{n,\omega_n}$, the coefficients on the coordinates in $\Delta_n$ are equal to $1$ and every other non-zero coefficient is equal to $2$.
\end{lemma}

\begin{proof}
For $i=n$, the first case in \cite[Definition~4.5]{CFL24} is empty.
The coefficient is $1$ precisely for $j=\bar\ell$, hence on $\Delta_n$,
and every other non-zero coefficient is $2$.
\end{proof}

\subsection{Makhlin's polytope}

Let
\[
                  \lambda=\sum_{i=1}^n a_i\omega_i^{B_n}.
\]
Let \(P_{\Mak}^{B_n}(\lambda)\) be the polytope defined by Makhlin's inequalities \cite[Section~1 and Remark~2.3]{Mak19}; its integral points form the set \(\Pi_\lambda\). In the orthonormal basis $\eps_1,\ldots,\eps_n$, one has
\[
 \omega_i^{B_n}=\eps_1+\cdots+\eps_i\quad(i<n),
 \qquad
 \omega_n^{B_n}=\frac12(\eps_1+\cdots+\eps_n).
\]
Writing $\lambda=\sum_{i=1}^n\lambda_i\eps_i$, we have
\begin{equation}\label{eq:lambdaeps}
             \lambda_i=\sum_{j=i}^{n-1}a_j+\frac{a_n}{2}.
\end{equation}

We identify the positive roots of types $B_n$ and $C_n$ by
\[
 \eps_i-\eps_j\longleftrightarrow\eps_i-\eps_j,\qquad
 \eps_i+\eps_j\longleftrightarrow\eps_i+\eps_j,\qquad
 \eps_i\longleftrightarrow2\eps_i.
\]
This identifies the triangular coordinates and the corresponding Dyck paths.

Makhlin's inequalities are indexed by these paths: in the inequality attached to a path, every coordinate indexed by a short root of $B_n$ occurs with coefficient $1/2$, and every other coordinate with coefficient $1$ \cite[Section~1]{Mak19}.

Define a diagonal map $D_n$ on the type $C_n$ FFLV coordinates by
\begin{equation}\label{eq:D}
 D_n(e_{a,b})=
 \begin{cases}
 2e_{a,\bar a},&b=\bar a,\\
 e_{a,b},&b\neq\bar a.
 \end{cases}
\end{equation}
Set
\begin{equation}\label{eq:mu}
             \mu=2\sum_{i=1}^{n-1}a_i\omega_i^{C_n}+a_n\omega_n^{C_n}.
\end{equation}
For a Dyck path $\mathbf d$, write $S_B(q,\mathbf d)\leq M_B(\lambda,\mathbf d)$ and $S_C(p,\mathbf d)\leq M_C(\mu,\mathbf d)$ for the corresponding type $B$ and $C$ inequalities.
The following makes the diagonal relation of \cite[Remark~2.3]{Mak19} explicit in our normalization.
\begin{proposition}\label{prop:MakhlinC}
With the above notation,
\[
       P_{\Mak}^{B_n}(\lambda)
            =\frac12D_nP_{\FFLV}^{C_n}(\mu).
\]
\end{proposition}

\begin{proof}
Put $p=2D_n^{-1}q$. By the definition in \cite[Section~1]{Mak19}, for every corresponding Dyck path $\mathbf d$ one has
\[
             S_C(p,\mathbf d)=2S_B(q,\mathbf d).
\]
Here $p_\beta=2q_\beta$ off the long-root boundary and $p_{2\eps_a}=q_{\eps_a}$, giving the coefficient $\frac12q_{\eps_a}$ after division by $2$. Moreover, \eqref{eq:lambdaeps} and \eqref{eq:mu} give
\[
             M_C(\mu,\mathbf d)=2M_B(\lambda,\mathbf d).
\]
Thus the inequalities are equivalent.
\end{proof}

\subsection{The spin direction}

In the common $\eps$-coordinates, \eqref{eq:lambdaeps} and \eqref{eq:mu} give
\[
                              \mu=2\lambda.
\]
Thus $\omega_n^{B_n}$ maps to $\omega_n^{C_n}$. The degree $2\omega_n$ in \cite[Lemma~5.1]{Mak19} is a lattice-point phenomenon.

\section{The $B/C$ similarity}\label{sec:similarity}

We use the $B/C$ similarity from \cite{Kas96,CFL25}. Put
\begin{equation}\label{eq:di}
 d_i=\begin{cases}
 2,&i<m,\\
 1,&i=m.
 \end{cases}
\end{equation}
For a reduced word $\mathbf i=(i_1,\ldots,i_L)$ define
\[
 \Gamma_{\mathbf i}^{B,C}(x_1,\ldots,x_L)
       =(d_{i_1}x_1,\ldots,d_{i_L}x_L).
\]
By \cite[Theorem~3.2]{CFL25}, this identifies the type $B_m$ and type $C_m$ string cones for every reduced word of the longest element.

Let $A^B=(a^B_{ij})$ and $A^C=(a^C_{ij})$ be the Cartan matrices with the common numbering. Then
\begin{equation}\label{eq:cartan}
                     a^C_{ij}=\frac{d_i}{d_j}a^B_{ij}.
\end{equation}
At the double edge, \eqref{eq:cartan} exchanges $-1$ and $-2$.

\begin{proposition}\label{prop:weighted}
Let $\widehat{\mathbf i}=(i_1,\ldots,i_M)$ be a reduced word of the longest element and
\[
                    \eta^B=\sum_{i=1}^m c_i\omega_i^{B_m}.
\]
Set
\[
                    \eta^C=\sum_{i=1}^m d_ic_i\omega_i^{C_m}.
\]
Then
\[
 \Gamma_{\widehat{\mathbf i}}^{B,C}
 \bigl(Q_{\widehat{\mathbf i}}^{B_m}(\eta^B)\bigr)
 =
 Q_{\widehat{\mathbf i}}^{C_m}(\eta^C).
\]
If $\mathbf i=(i_1,\ldots,i_L)$ is a prefix of $\widehat{\mathbf i}$, then
\[
 \Gamma_{\mathbf i}^{B,C}
 \bigl(Q_{\mathbf i}^{B_m}(\eta^B)\bigr)
 =
 Q_{\mathbf i}^{C_m}(\eta^C).
\]
\end{proposition}

\begin{proof}
By \cite[Theorem~3.2]{CFL25}, only the highest-weight inequalities remain. Let $x=(x_1,\ldots,x_M)$ satisfy the type $B_m$ inequalities and set
\[
                         y_k=d_{i_k}x_k.
\]
For the $k$-th inequality in \eqref{eq:littelmann} we obtain
\begin{align*}
 y_k
 &\leq d_{i_k}c_{i_k}
      -\sum_{j>k}d_{i_k}a^B_{i_k,i_j}x_j\\
 &=d_{i_k}c_{i_k}
      -\sum_{j>k}\frac{d_{i_k}}{d_{i_j}}a^B_{i_k,i_j}y_j\\
 &=d_{i_k}c_{i_k}
      -\sum_{j>k}a^C_{i_k,i_j}y_j,
\end{align*}
by \eqref{eq:cartan}. Since
\[
 \langle\eta^C,(\alpha_i^C)^\vee\rangle=d_ic_i,
\]
these are precisely the type $C_m$ highest-weight inequalities.

For the second claim, let $\pi_L$ denote the projection to the first
$L$ coordinates. Since $\Gamma_{\widehat{\mathbf i}}^{B,C}$ is diagonal,
\[
 \pi_L\Gamma_{\widehat{\mathbf i}}^{B,C}
   =\Gamma_{\mathbf i}^{B,C}\pi_L.
\]
The claim follows from the definition of the Demazure string polytopes.
\end{proof}

Apply the proposition to a reduced extension of $\ww_n$ and to
\[
       \widetilde\lambda
          =\sum_{i=1}^n a_i\widetilde\omega_{2i-1}^{B_m}.
\]
Since $2i-1<m$ for $i<n$ and $2n-1=m$, the transformed weight is
\begin{equation}\label{eq:mut}
 \widetilde\mu
  =2\sum_{i=1}^{n-1}a_i\widetilde\omega_{2i-1}^{C_m}
       +a_n\widetilde\omega_m^{C_m}.
\end{equation}
This is the lift of $\mu$ from \eqref{eq:mu}; hence
\begin{equation}\label{eq:BtoC}
 \Gamma_n Q_{\ww_n}^{B_m}(\widetilde\lambda)
       =Q_{\ww_n}^{C_m}(\widetilde\mu),
\end{equation}
where $\Gamma_n:=\Gamma_{\ww_n}^{B,C}$. Note that $\Gamma_n$ is not unimodular: by Lemma~\ref{lem:GammaD} below it rescales $N-n$ of the $N$ coordinates by $2$, hence changes the lattice by the index $2^{N-n}$. This is exactly the factor cancelled by the map $2D_n^{-1}$ of Proposition~\ref{prop:MakhlinC}.

\begin{example}\label{ex:weightsB3}
For $n=3$ the word is
\[
                 \ww_3=(5,4,5,3,4,5,2,3,1).
\]
Let
\[
      \eta^B=a_1\omega_1^{B_5}+a_2\omega_3^{B_5}+a_3\omega_5^{B_5}
\]
and denote the string coordinates by $t_1,\ldots,t_9$. The nine highest-weight inequalities are
\[
\begin{aligned}
 t_1&\leq a_3+2t_2-2t_3+2t_5-2t_6,\\
 t_2&\leq t_3+t_4-2t_5+t_6+t_8,\\
 t_3&\leq a_3+2t_5-2t_6,\\
 t_4&\leq a_2+t_5+t_7-2t_8,\\
 t_5&\leq t_6+t_8, & t_6&\leq a_3,\\
 t_7&\leq t_8+t_9, & t_8&\leq a_2,\qquad t_9\leq a_1.
\end{aligned}
\]
The similarity is
\[
 (x_1,\ldots,x_9)
  =(t_1,2t_2,t_3,2t_4,2t_5,t_6,2t_7,2t_8,2t_9).
\]
After substitution, the inequalities become
\[
\begin{aligned}
 x_1&\leq a_3+x_2-2x_3+x_5-2x_6,\\
 x_2&\leq2x_3+x_4-2x_5+2x_6+x_8,\\
 x_3&\leq a_3+x_5-2x_6,\\
 x_4&\leq2a_2+x_5+x_7-2x_8,\\
 x_5&\leq2x_6+x_8, & x_6&\leq a_3,\\
 x_7&\leq x_8+x_9, & x_8&\leq2a_2,\qquad x_9\leq2a_1.
\end{aligned}
\]
These are exactly the type $C_5$ inequalities for
\[
               2a_1\omega_1^{C_5}+2a_2\omega_3^{C_5}+a_3\omega_5^{C_5}.
\]
\end{example}

\begin{lemma}\label{lem:GammaD}
In the root-coordinate labeling of \cite[Section~4]{CFL24},
\[
                           \Gamma_n=2D_n^{-1}.
\]
\end{lemma}

\begin{proof}
Let $\iota(a,b)$ be the simple-reflection index of $e_{a,b}$ in \cite{CFL24}. By \cite[Notation~4.10]{CFL24},
\[
                   \iota(a,b)=m
                   \quad\Longleftrightarrow\quad
                   b=\bar a.
\]
Thus the $f_m$-coordinates are exactly
\[
                   e_{1,\bar1},\ldots,e_{n,\bar n}.
\]
Hence, by \eqref{eq:di} and the definition of $\Gamma_n$,
\[
 \Gamma_n(e_{a,b})=
 \begin{cases}
   e_{a,b},&b=\bar a,\\
   2e_{a,b},&b\neq\bar a.
 \end{cases}
\]
This equals $2D_n^{-1}(e_{a,b})$ by \eqref{eq:D}.
\end{proof}

\section{Proof of the main result}\label{sec:proof}

Let $q\in P_{\Mak}^{B_n}(\lambda)$. By Proposition~\ref{prop:MakhlinC},
\[
                   2D_n^{-1}q\in P_{\FFLV}^{C_n}(\mu).
\]
Theorem~\ref{thm:CFL} gives
\[
       X_n(2D_n^{-1}q)+t_{n,\mu}
          \in Q_{\ww_n}^{C_m}(\widetilde\mu).
\]
Applying $\Gamma_n^{-1}$ and Lemma~\ref{lem:GammaD} gives
\begin{equation}\label{eq:Phi}
 \begin{split}
 \Phi_{n,\lambda}(q)
 &=\Gamma_n^{-1}\bigl(X_n(2D_n^{-1}q)+t_{n,\mu}\bigr)\\
 &=D_nX_nD_n^{-1}q+\frac12D_nt_{n,\mu}.
 \end{split}
\end{equation}
Set
\begin{equation}\label{eq:Uc}
              U_n=D_nX_nD_n^{-1},
       \qquad c_\lambda=\frac12D_nt_{n,\mu}.
\end{equation}

It remains to check integrality.

\begin{proposition}\label{prop:U}
One has $U_n\in GL_N(\ZZ)$.
\end{proposition}

\begin{proof}
Only the columns indexed by $e_{a,\bar a}$ may contain denominators. By \eqref{eq:Xdiag},
\[
 \begin{split}
 U_n(e_{a,\bar a})
 &=D_nX_n\left(\frac12e_{a,\bar a}\right)\\
 &=-D_n\left(\frac12e_{a,\bar a}+\sum_{c<a}e_{c,\bar a}\right)\\
 &=-e_{a,\bar a}-\sum_{c<a}e_{c,\bar a}.
 \end{split}
\]
Hence these columns are integral. On the remaining basis vectors $D_n^{-1}$ acts trivially, so $U_n\in M_N(\ZZ)$.

Since $U_n$ is conjugate to $X_n$,
\[
                 \det U_n=\det X_n=\pm1,
\]
so $U_n\in GL_N(\ZZ)$.
\end{proof}

\begin{proposition}\label{prop:c}
For every dominant integral weight $\lambda$ of type $B_n$, the vector $c_\lambda$ belongs to $\ZZ^N$. Moreover
\[
                         c_{\lambda+\nu}=c_\lambda+c_\nu.
\]
\end{proposition}

\begin{proof}
By the linearity of the translation vector in \cite{CFL24} and \eqref{eq:mu},
\[
 t_{n,\mu}=2\sum_{i=1}^{n-1}a_it_{n,\omega_i}+a_nt_{n,\omega_n}.
\]
Hence
\[
 c_\lambda=\sum_{i=1}^{n-1}a_iD_nt_{n,\omega_i}
             +\frac{a_n}{2}D_nt_{n,\omega_n}.
\]
Lemma~\ref{lem:translation} gives $D_nt_{n,\omega_n}\in2\ZZ^N$, so $c_\lambda$ is integral. Additivity is immediate.
\end{proof}

\begin{proof}[Proof of Theorem~\ref{thm:main}]
Proposition~\ref{prop:MakhlinC}, Theorem~\ref{thm:CFL} and Proposition~\ref{prop:weighted} show that \eqref{eq:Phi} is an affine bijection
\[
 P_{\Mak}^{B_n}(\lambda)
       \longrightarrow Q_{\ww_n}^{B_m}(\widetilde\lambda).
\]
Propositions~\ref{prop:U} and \ref{prop:c} give unimodularity. Additivity follows since $U_n$ is independent of $\lambda$ and $c_\lambda$ is additive.
\end{proof}

\begin{remark}\label{rem:cancellation}
Lemma~\ref{lem:GammaD} gives the cancellation of the two diagonal rescalings; the linear part is $D_nX_nD_n^{-1}$.
\end{remark}

\section{Lattice points and semigroups}\label{sec:semigroup}

Let
\[
 S_{\Mak}(\lambda)=P_{\Mak}^{B_n}(\lambda)\cap\ZZ^N,
 \qquad
 S_{\Str}(\widetilde\lambda)
 =Q_{\ww_n}^{B_m}(\widetilde\lambda)\cap\ZZ^N.
\]
The sets $S_{\Mak}(\lambda)$ and $S_{\Str}(\widetilde\lambda)$ parametrize Makhlin's monomial basis \cite[Theorem~1.1 and Corollary~2.2]{Mak19} and the Demazure string basis \cite{Lit98,BZ01,Cal02}, respectively. Theorem~\ref{thm:main} gives
\begin{equation}\label{eq:lattice}
         \Phi_{n,\lambda}\bigl(S_{\Mak}(\lambda)\bigr)
                  =S_{\Str}(\widetilde\lambda).
\end{equation}

Since $U_n$ is independent of $\lambda$ and $c_\lambda$ is additive,
\begin{equation}\label{eq:additive}
 \Phi_{n,\lambda+\nu}(x+y)
       =\Phi_{n,\lambda}(x)+\Phi_{n,\nu}(y).
\end{equation}
Define
\[
 \Sigma_{\Mak}
   =\{(\lambda,x)\mid \lambda\in P^+(B_n),\ x\in S_{\Mak}(\lambda)\}
\]
and
\[
 \Sigma_{\Str}
   =\{(\lambda,y)\mid \lambda\in P^+(B_n),\
          y\in S_{\Str}(\widetilde\lambda)\}.
\]
Both are indeed semigroups: the cone parts of $P_{\Mak}^{B_n}(\lambda)$ and of
$Q_{\ww_n}^{B_m}(\widetilde\lambda)$ do not depend on the weight, and the remaining
bounds, $M_B(\lambda,\mathbf d)$ in Section~\ref{sec:preliminaries} resp.\ the
right hand sides of \eqref{eq:littelmann}, are linear in it. Since
$\widetilde\lambda+\widetilde\nu=\widetilde{\lambda+\nu}$, this gives
\[
 S_{\Mak}(\lambda)+S_{\Mak}(\nu)\subseteq S_{\Mak}(\lambda+\nu),
 \qquad
 S_{\Str}(\widetilde\lambda)+S_{\Str}(\widetilde\nu)
      \subseteq S_{\Str}(\widetilde{\lambda+\nu}).
\]
\begin{corollary}\label{cor:semigroup}
The map
\[
        (\lambda,x)\longmapsto(\lambda,\Phi_{n,\lambda}(x))
\]
is an isomorphism of semigroups $\Sigma_{\Mak}\simeq\Sigma_{\Str}$.
\end{corollary}

\begin{proof}
Bijectivity on each homogeneous component follows from \eqref{eq:lattice}; compatibility with addition is \eqref{eq:additive}.
\end{proof}

Let $\lambda$ be neither fundamental nor $2\omega_n$, and let $l=\min\{i:a_i\neq0\}$. By \cite[Lemma~5.1]{Mak19}, every lattice point splits off a point of $S_{\Mak}(\omega_l)$ for $l<n$, and of $S_{\Mak}(2\omega_n)$ for $l=n$. Hence $\Sigma_{\Mak}$ is generated in degrees
\[
                       \omega_1,\ldots,\omega_n,\ 2\omega_n.
\]
Corollary~\ref{cor:semigroup} gives the same for $\Sigma_{\Str}$.

\begin{corollary}\label{cor:generators}
The semigroup $\Sigma_{\Str}$ is generated by its homogeneous components of degrees
\[
                         \omega_1,\ldots,\omega_n,\ 2\omega_n.
\]
Equivalently, the corresponding lattice-point sets are
\[
 S_{\Str}(\widetilde\omega_1),\,
 S_{\Str}(\widetilde\omega_3),\ldots,\,
 S_{\Str}(\widetilde\omega_{2n-1}),\,
 S_{\Str}(2\widetilde\omega_{2n-1}).
\]
\end{corollary}

Although
\[
 P_{\Mak}^{B_n}(2\omega_n)=2P_{\Mak}^{B_n}(\omega_n),
\]
the degree $2\omega_n$ occurs in the decomposition of \cite[Lemma~5.1]{Mak19}.

\printbibliography

\end{document}